\documentclass[12pt]{article}

\usepackage[utf8]{inputenc}

\usepackage{a4,amsmath,amsfonts,amsthm,latexsym,amssymb,graphicx}
\usepackage{fullpage}
\usepackage{tocloft}
\usepackage{bm}
\usepackage{enumerate}
\usepackage{enumitem}
\usepackage{bbm}
\usepackage{mathtools}
\usepackage{tikz}
\usepackage{tikz-cd}
\usepackage{pgf}
\usepackage{hhline}
\usepackage{float}
\usepackage{url}
\usepackage[symbol]{footmisc}

\DeclareMathSymbol{\Q}{\mathalpha}{AMSb}{"51}
\DeclareMathSymbol{\R}{\mathalpha}{AMSb}{"52}
\DeclareMathSymbol{\Z}{\mathalpha}{AMSb}{"5A}
\DeclareMathSymbol{\N}{\mathalpha}{AMSb}{"4E}
\DeclareMathSymbol{\C}{\mathalpha}{AMSb}{"43}

\newcommand{\Ann}{\ensuremath{\text{\rm Ann}}}
\newcommand{\Per}{\ensuremath{\text{\rm Per}}}
\newcommand{\ord}{\ensuremath{\text{\rm ord}}}
\newcommand{\supp}{\ensuremath{\text{\rm supp}}}
\newcommand{\Res}{\ensuremath{\text{\rm Res}}}
\newcommand{\conv}{\ensuremath{\text{\rm conv}}}
\newcommand{\fib}[2]{\ensuremath{\text{\rm fib}}_{#1}(#2)}
\newcommand{\A}{\ensuremath{\mathcal{A}}}

\newcommand{\Patt}[2]{\ensuremath{\mathcal{L}_{#2}(#1)}}
\newcommand{\Lang}[1]{\ensuremath{\mathcal{L}(#1)}}

\renewcommand{\vec}[1]{\mathbf{#1}}

\newtheorem*{theorem*}{Theorem}
\newtheorem*{corollary*}{Corollary}

\newtheorem{theorem}{Theorem}
\newtheorem{lemma}[theorem]{Lemma}

\newtheorem*{claim*}{Claim}

\newtheorem{manualtheoreminner}{Theorem}
\newenvironment{manualtheorem}[1]{%
  \renewcommand\themanualtheoreminner{#1}%
  \manualtheoreminner
}{\endmanualtheoreminner}

\newtheorem{reformulation}{Reformulation of Theorem}

\theoremstyle{definition}
\newtheorem{definition}[theorem]{Definition}

\newtheorem{example}[theorem]{Example}

\title{On perfect coverings of two-dimensional grids}
\author{Elias Heikkilä, Pyry Herva and Jarkko Kari}
\date{}

\begin{document}

\maketitle

\begin{abstract}
    \noindent
    We study perfect multiple coverings in translation invariant graphs with vertex set $\Z^2$ using an algebraic approach.
    In this approach we consider any such covering as a two-dimensional binary configuration which we then express as a two-variate formal power series.
    Using known results, we conclude that any perfect multiple covering has a non-trivial periodizer, that is, there exists a non-zero polynomial whose formal product with the power series presenting the covering is a two-periodic configuration.
    If a non-trivial periodizer has line polynomial factors in at most one direction, then the configuration is known to
    be periodic.
    Using this result we find many setups where perfect multiple coverings of infinite grids are necessarily periodic.
    We also consider some algorithmic questions on finding perfect multiple coverings.
\end{abstract}

\section{Introduction and preliminaries}

A perfect multiple covering in a graph is a set of vertices, a code, such that
the number of codewords in the neighborhood of an arbitrary vertex depends only on whether the vertex is in the code or not.
In this paper we study these codes on translation invariant graphs with the vertex set $\Z^2$.
We present codes as two-dimensional binary configurations 
and observe that
the perfect covering condition provides an algebraic condition that can be treated with the
algebraic tools developed in~\cite{icalp}. We focus on periodic codes and, in particular, study
setups where all codes are necessarily periodic.
The approach we take was initially mentioned in an example in the survey \cite{surveyjarkko} by the third author, and
considered in the Master's thesis \cite{Heikkila} by the first author.

We start by giving the basic definitions, presenting the aforementioned algebraic approach and stating some past results relevant to us.
In Section 2 we describe an algorithm to find the line polynomial factors of any given (Laurent) polynomial.
In Section 3 we formally define the perfect multiple coverings in graphs and prove some periodicity results concerning them.
  We give new algebraic proofs of some known results concerning perfect multiple coverings on the infinite square grid and on the triangular grid ~\cite{Axenovich,puzynina2}, and provide a new result on the forced periodicity of such coverings on 
  the king grid.
Furthermore, we generalize the definition of perfect coverings for two-dimensional binary configurations with respect to different neighborhoods and covering constants.
In Section 4 we consider some algorithmic questions concerning perfect coverings.
Using a standard argument by H. Wang we show that under certain constraints it is algorithmically decidable to determine whether there exist any perfect coverings with given neighborhood and given covering constants.

\subsection*{Configurations, periodicity, finite patterns and subshifts}


A $d$-dimensional \emph{configuration} is a coloring of the infinite grid $\Z^d$ using finitely many colors, that is, an element of $\A^{\Z^d}$ which we call the $d$-dimensional \emph{configuration space}
where $\A$ is some finite alphabet.
For a configuration $c$ we let $c_{\vec{u}} = c(\vec{u})$ to be the symbol or color that $c$ has in cell $\vec{u}$.
The \emph{translation} $\tau^{\vec{t}}$ by a vector $\vec{t} \in \Z^d$ shifts a configuration $c$ such that $\tau^{\vec{t}}(c)_{\vec{u}} = c_{\vec{u} - \vec{t}}$ for all $\vec{u}\in\Z^d$.
A configuration $c$ is \emph{$\vec{t}$-periodic} if $\tau^{\vec{t}}(c) = c$
and $c$ is \emph{periodic} if $c$ is $\vec{t}$-periodic for some non-zero $\vec{t} \in \Z^d$.
We also say that a configuration $c$ is \emph{periodic in direction} $\vec{v} \in \Z^d\setminus\{\vec{0}\}$ if $c$ is $k \vec{v}$-periodic for some $k \in \Q$.
A $d$-dimensional configuration $c$
is \emph{strongly periodic} if it has $d$
linearly independent vectors of periodicity.
Strongly periodic configurations are then periodic in all directions.
Two-dimensional strongly periodic configurations are called \emph{two-periodic}.


A finite \emph{pattern} is an assignment of symbols on some finite shape $D \subseteq \Z^d$, that is, an element of $\A^D$ where $\A$ is some fixed alphabet.
In particular, the finite patterns in $\A^D$ are called \emph{$D$-patterns}.
Let us denote by $\A^*$ the set of all finite patterns over alphabet $\A$ where the dimension $d$ is known from the context.
A finite pattern $p \in \A^D$ \emph{appears} in a configuration $c \in \A^{\Z^d}$ if $\tau^{\vec{t}}(c)|_D = p$ for some $\vec{t} \in \Z^d$.
A configuration $c$ \emph{contains} the pattern $p$ if it appears in $c$.
For a fixed shape $D$, the set of all $D$-patterns that appear in $c$ is the set $\Patt{c}{D} = \{ \tau^{\vec{t}}(c)|_D \mid \vec{t} \in \Z^d \}$ and the set of all finite patterns in $c$
is denoted by $\Lang{c}$ which we call the \emph{language of $c$}. For a set $\mathcal{S} \subseteq \A^{\Z^d}$ of configurations
we define $\Patt{\mathcal{S}}{D}$ and  $\Lang{\mathcal{S}}$ as the unions of $\Patt{c}{D}$ and $\Lang{c}$ over all 
$c\in \mathcal{S}$, respectively.


Let us review some basic concepts of symbolic dynamics we need. For a reference see \emph{e.g.} \cite{tullio,kurka,lindmarcus}.
The configuration space $\A^{\Z^d}$ can be made a compact topological space by endowing $\A$ with the discrete topology and considering the product topology it induces on $\A^{\Z^d}$ -- the \emph{prodiscrete topology}.
This topology is
induced by a metric where two configurations are close if they agree on a large area around the origin.
Thus $\A^{\Z^d}$ is a compact metric space.

A subset $\mathcal{S} \subseteq \A^{\Z^d}$ of the configuration space is a \emph{subshift} if it is topologically closed and translation-invariant meaning that if $c \in \mathcal{S}$ then for any $\vec{t} \in \Z^d$ also $\tau^{\vec{t}}(c) \in \mathcal{S}$.
Equivalently we can define subshifts using forbidden patterns:
Given a set $F \subseteq \A^*$ of \emph{forbidden} finite patterns, the set
$$
X_F=\{c \in \A^{\Z^d} \mid  \Lang{c} \cap F = \emptyset \}
$$
of configurations that avoid all forbidden patterns
is a subshift, and every subshift is obtained by forbidding some set of finite patterns.
If $F \subseteq \A^*$ is finite then we say that $X_F$ is a \emph{subshift of finite type} (SFT).

The \emph{orbit} of a configuration $c$ is the set $\mathcal{O}(c) = \{ \tau^{\vec{t}}(c) \mid \vec{t} \in \Z^d \}$ of its every translate.
The
\emph{orbit closure} $\overline{\mathcal{O}(c)}$ is the topological closure of its orbit under the prodiscrete topology.
The orbit closure of a configuration $c$ is the smallest subshift that contains $c$. It consists of all configurations $c'$ such that $\Lang{c'}\subseteq \Lang{c}$.

\subsection*{The algebraic approach}

To present a configuration $c \in \A^{\Z^d}$ algebraically we make the assumption that $\A \subseteq \Z$. Then we identify the configuration $c$ with the formal power series
$$
c(X) = \sum_{\vec{u} \in \Z^d} c_{\vec{u}} X^{\vec{u}}
$$
over $d$ variables $x_1, \ldots , x_d$ where we have denoted $X = (x_1, \ldots , x_d)$ and $X^{\vec{u}} = x_1^{u_1} \cdots x_d^{u_d}$ for any $\vec{u} = (u_1,\ldots ,u_d) \in \Z^d$.
For $d=2$ we usually denote $X=(x,y)$.
More generally we study the set of all formal power series over $d$ variables $x_1,\ldots ,x_d$ with complex coefficients which we denote by $\C[[X^{\pm 1}]] = \C[[x_1^{\pm1}, \ldots , x_d^{\pm1}]]$. A power series is \emph{finitary} if it has only finitely many different coefficients and \emph{integral} if its coefficients are all integers.
Thus we identify configurations with finitary and integral power series.

We also use Laurent polynomials which we call from now on simply polynomials.
We use the term ``proper'' when we talk about proper ({\it i.e.}, non-Laurent) polynomials.
Let us denote by $\C[X^{\pm 1}] = \C[x_1^{\pm1}, \ldots , x_d^{\pm1}]$ the set of all (Laurent) polynomials over $d$ variables $x_1,\ldots ,x_d$ with complex coefficients, which is the \emph{Laurent polynomial ring}.
We say that two polynomials have no common factors if all of their common factors are units and that they have a common factor if they have a non--unit common factor.

A product of a polynomial and a power series is well defined.
We say that a polynomial $f=f(X)$ \emph{annihilates} (or is an annihilator of) a power series $c=c(X)$ if $fc=0$, that is, if their product is the zero power series.
We say that a formal power series
$c=c(X)$ is \emph{periodic} if it is annihilated by a \emph{difference polynomial} $X^{\vec{t}}-1$ where $\vec{t}$ is non-zero.
Note that this definition is consistent with the definition of periodicity of configurations defined above. Indeed if $c=c(X)$ is a configuration then multiplying it by a monomial $X^{\vec{t}}$ produces the translated configuration $\tau^{\vec{t}}(c)$ and hence $c$ is $\vec{t}$-periodic if and only if $c = \tau^{\vec{t}}(c) = X^{\vec{t}} c$, which is equivalent to $(X^{\vec{t}}-1)c=0$.
So it is natural to study the \emph{annihilator ideal}
$$
\Ann(c) = \{ f \in \C[X^{\pm 1}] \mid fc = 0 \}
$$
of a power series $c \in \C[[X^{\pm 1}]]$, which indeed is an ideal of the Laurent polynomial ring.
Hence the question whether a configuration (or any formal power series) is periodic is equivalent to asking whether its annihilator ideal contains a difference polynomial.
Another useful ideal that we study is the \emph{periodizer ideal}
$$
\Per(c) = \{ f \in \C[X^{\pm 1}] \mid fc \text{ is strongly periodic} \}.
$$
Note that clearly $\Ann(c)$ is a subset of $\Per(c)$.
Note also that a configuration $c$ has a non-trivial (= non-zero) annihilator if and only if it has a non-trivial periodizer.
The following theorem states that if a configuration has a non-trivial periodizer then it has in fact an annihilator of a particular simple form -- a product of difference polynomials.

\begin{theorem}[\cite{icalp}]
    \label{special annihilator}
    Let $c$ be a configuration in any dimension that has a non-trivial periodizer.
     Then there exist pairwise linearly independent vectors $\vec{t}_1, \ldots ,\vec{t}_m$ with $m \geq 1$ such that
    $$
    (X^{\vec{t}_1}-1) \cdots (X^{\vec{t}_m}-1) \in \Ann(c).
    $$
\end{theorem}

\subsection*{Line polynomials}

The \emph{support} of a power series $c = \sum_{\vec{u}  \in \Z^d} c_{\vec{u}} X^{\vec{u}}$ is the set $\supp(c) = \{ \vec{u} \in \Z^d \mid c_{\vec{u}} \neq 0 \}$.
Thus a polynomial is a power series with a finite support.
A \emph{line polynomial} is a polynomial whose support contains at least two points and the points of the support lie on a unique line.
In other words, a polynomial $f$ is a line polynomial if it is not a monomial
and there exist vectors $\vec{u}, \vec{v} \in \Z^d$ such that $\supp(f) \subseteq \vec{u} + \Q \vec{v}$.
In this case we say that $f$ is a line polynomial in direction $\vec{v}$.
We say that non-zero vectors $\vec{v},\vec{v}'\in\Z^d$ are \emph{parallel} if $\vec{v}'\in \Q \vec{v}$, and
clearly then a line polynomial in direction $\vec{v}$ is also a line polynomial in any parallel direction.
A vector $\vec{v}\in\Z^d$ is \emph{primitive} if its components are pairwise relatively prime. If
$\vec{v}$ is primitive then $\Q \vec{v}\cap\Z^d = \Z \vec{v}$.
For any non-zero $\vec{v}\in\Z^d$ there exists a parallel primitive vector $\vec{v}'\in\Z^d$. It follows that we may assume the vector $\vec{v}$ in the definition of a line polynomial $f$ to be primitive so that $\supp(f) \subseteq \vec{u} + \Z \vec{v}$.
In the following our preferred presentations of directions are in terms of primitive vectors.


Any line polynomial
$\phi$ in a (primitive) direction $\vec{v}$ can be written uniquely in the form
$$
\phi =
X^{\vec{u}}(a_0 + a_1 X^{\vec{v}} + \ldots + a_n X^{n \vec{v}}) = X^{\vec{u}}(a_0 + a_1 t + \ldots + a_n t^n)
$$
where $\vec{u} \in \Z^d, n \geq 1 , a_0 \neq 0 , a_n \neq 0$ and $t = X^{\vec{v}}$.
Let us call
the single variable proper polynomial $a_0 + a_1 t + \ldots + a_n t^n \in \C[t]$ the \emph{normal form} of $\phi$.
Moreover, for a monomial $a X^{\vec{u}}$ we define its normal form to be $a$.
Thus two line polynomials in the direction $\vec{v}$
have the same normal form if and only if
they are the same polynomial up to multiplication by $X^{\vec{u}}$, for some $\vec{u}\in\Z^d$.

Difference polynomials are line polynomials and hence the annihilator provided by Theorem \ref{special annihilator} is a product of line polynomials.
Annihilation by a difference polynomial means periodicity. More generally,
annihilation of a configuration $c$ by a line polynomial in a primitive direction $\vec{v}$
can be understood as the annihilation of the one-dimensional \emph{$\vec{v}$-fibers} $\sum_{k \in \Z} c_{\vec{u} + k \vec{v}} X^{\vec{u}+k\vec{v}}$ of $c$ in
direction $\vec{v}$, and since annihilation in the one-dimensional setting implies periodicity we conclude that a configuration is periodic if and only if it is annihilated by a line polynomial.
It is known that if $c$ has a periodizer with line polynomial factors in at most one direction then $c$ is periodic:

\begin{theorem}[\cite{fullproofs}] \label{theorem on line polynomial factors}
    Let $c$ be a two-dimensional configuration and $f \in \Per(c)$. Then the following conditions hold.
    \begin{itemize}
        \item If $f$ does not have any line polynomial factors then $c$ is two-periodic.
        \item If all line polynomial factors of $f$ are in the same direction then $c$ is periodic in this direction.
    \end{itemize}
\end{theorem}

\noindent
\emph{Proof sketch.}
The periodizer ideal $\Per(c)$ is a principal ideal generated by a polynomial $g = \phi_1 \cdots \phi_m$  where $\phi_1, \ldots ,\phi_m$ are line polynomials in pairwise non-parallel directions~\cite{fullproofs}.
Because $f\in\Per(c)$ we know that $g$ divides $f$.
If $f$ does not have any line polynomial factors then $g = 1$ and thus $c = gc$ is two-periodic.
 If $f$ has line polynomial factors and they are in the same primitive
direction $\vec{v}$ then $g$ is a line polynomial in this direction.
Since $gc$ is two-periodic it is annihilated by $(X^{k \vec{v}} - 1)$ for some $k \in \Z$.
Then the configuration $c$ is annihilated by the line polynomial $(X^{k \vec{v}} - 1)g$ 
in direction $\vec{v}$.
We conclude that $c$ is periodic in direction $\vec{v}$.
\qed

\medskip

\noindent
(See the Appendix for an alternative proof that mimics the usage of resultants in~\cite{karimoutot},
instead of relying on the structure of the ideal $\Per(c)$.)

\section{Line polynomial factors}

The open and closed \emph{discrete half planes} determined by a non-zero vector $\vec{v} \in \Z^2$ are the sets $H_{\vec{v}} = \{ \vec{u} \in \Z^2 \mid \langle \vec{u} , \vec{v} ^{\perp} \rangle > 0 \}$ and $\overline{H}_{\vec{v}} = \{ \vec{u} \in \Z^2 \mid \langle \vec{u} , \vec{v} ^{\perp} \rangle \geq 0 \}$, respectively, where $\vec{v}^{\perp} = (v_2,-v_1)$ is orthogonal to $\vec{v} = (v_1,v_2)$.
Let us also denote by $l_{\vec{v}} = \overline{H}_{\vec{v}} \setminus H_{\vec{v}}$ the discrete line parallel to $\vec{v}$ that goes through the origin.
In other words, the half plane determined by $\vec{v}$ is the half plane ``to the right'' of the line $l_{\vec{v}}$ when moving along the line in the direction of $\vec{v}$.
We say that a finite set $D \subseteq \Z^2$ has an \emph{outer edge} in direction $\vec{v}$ if there exists a vector $\vec{t} \in \Z^2$ such that $D \subseteq \overline{H}_{\vec{v}} + \vec{t}$ and $|D \cap (l_{\vec{v}} + \vec{t})| \geq 2$.
We then call $D \cap (l_{\vec{v}} + \vec{t})$ an outer edge of $D$ in direction $\vec{v}$.
An outer edge corresponding to $\vec{v}$ means that the convex hull of $D$  has an edge in direction $\vec{v}$
in the clockwise orientation around $D$.

If a finite non-empty set $D$ does not have an outer edge in direction $\vec{v}$ then there exists a vector $\vec{t} \in \Z^2$ such that $D \subseteq \overline{H}_{\vec{v}} + \vec{t}$ and $|D \cap (l_{\vec{v}} + \vec{t})| = 1$ and then we say that $D$ has a vertex in direction $\vec{v}$ and we call $D \cap (l_{\vec{v}} + \vec{t})$ a vertex of $D$ in direction $\vec{v}$.
We say that a polynomial $f$ has an outer edge or a vertex in direction $\vec{v}$ if its support has an outer edge or a vertex in direction $\vec{v}$, respectively.
Note that every finite shape $D$ has either an edge or a vertex in any non-zero direction. Note also that in this context
directions $\vec{v}$ and $-\vec{v}$ are not the same: a shape may have an outer edge in direction $\vec{v}$
but no outer edge in direction $-\vec{v}$.
The following lemma shows that a polynomial can have line polynomial factors only in the directions of its outer edges.

\begin{lemma}[\cite{karimoutot}]
\label{lemma1}
    Let $f$ be a non-zero polynomial with a line polynomial factor in direction $\vec{v}$. Then $f$ has outer edges in directions $\vec{v}$ and $-\vec{v}$.
\end{lemma}

\noindent

Let $\vec{v} \in \Z^2 \setminus \{ \vec{0} \}$ be any non-zero primitive vector
and let $f = \sum f_{\vec{u}} X^{\vec{u}}$ be a polynomial.
Recall that a \emph{$\vec{v}$-fiber} of $f$ is a polynomial
of the form $\sum_{k \in \Z} f_{\vec{u} + k \vec{v}} X^{\vec{u} + k \vec{v}}$ for some $\vec{u}\in\Z^2$.
Thus a non-zero $\vec{v}$-fiber of a polynomial is either a line polynomial or a monomial.
Let us denote by $\mathcal{F}_{\vec{v}}(f)$ the set of different normal forms of all non-zero
$\vec{v}$-fibers of a polynomial $f$, which is thus a finite set.
The following simple example illustrates the concept of fibers and their normal forms.

\begin{figure}
	\centering
	\begin{tikzpicture}[scale=0.6]
		\draw[fill=red] (1,0) circle(3pt);
		\node at (1,0.3) {\tiny $3x$};
		
		\draw[fill=blue] (0,1) circle(3pt);
		\node at (0,1.3) {\tiny $y$};
		
		\draw[fill=blue] (1,2) circle(3pt);
		\node at (1,2.35) {\tiny $xy^2$};
		
		\draw[fill=green] (1,1) circle(3pt);
		\node at (0.9,1.4) {\tiny $xy$};
		
		\draw[fill=green] (3,3) circle(3pt);
		\node at (2.9,3.4) {\tiny $x^3y^3$};
		
		\draw[fill=green] (4,4) circle(3pt);
		\node at (3.9,4.4) {\tiny $x^4y^4$};
		
		\draw[] (-0.5,0.5) -- (1.5,2.5);
		
		\draw[] (0.5,0.5) -- (4.5,4.5);
		
		\draw[] (0.5,-0.5) -- (1.5,0.5);
	\end{tikzpicture}
	\caption{The support of $f = 3x + y + xy^2 +xy + x^3y^3 + x^4y^4$ and its different $(1,1)$-fibers.}
	\label{havainnollistus}
\end{figure}
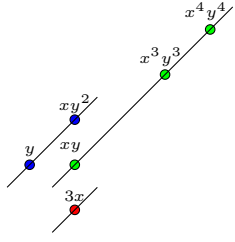

\begin{example}
    Let us determine the set $\mathcal{F}_{\vec{v}}(f)$ for $f = f(X) = f(x,y) = 3x + y + xy^2 +xy + x^3y^3 + x^4y^4$
    and $\vec{v} = (1,1)$.
    By grouping the terms we can write
    $$
    f = 3x + y(1 + xy) + xy(1 + x^2y^2 + x^3y^3)
    = X^{(1,0)} \cdot 3 + X^{(0,1)}(1 + t) + X^{(1,1)}(1 + t^2 + t^3)
    $$
    where $t = X^{(1,1)} = xy$.
    Hence
    $\mathcal{F}_{\vec{v}}(f) =
    \{ 3, 1 + t, 1 + t^2 + t^3 \}$.
    See Figure \ref{havainnollistus} for a pictorial illustration.
    \qed
\end{example}




\noindent
As noticed in the example above, polynomials are linear combinations of their fibers: for any polynomial $f$ and any non-zero primitive vector $\vec{v}$ we can write
$$
f = X^{\vec{u}_1} \psi_1 + \ldots + X^{\vec{u}_n} \psi_n
$$
for some $\vec{u}_1, \ldots , \vec{u}_n \in \Z^2$ where $\psi_1, \ldots , \psi_n\in \mathcal{F}_{\vec{v}}(f)$.
We use this in the proof of the next theorem.

\begin{theorem} \label{theorem1}
	A polynomial $f$ has a line polynomial factor in direction $\vec{v}$ if and only if the polynomials in $\mathcal{F}_{\vec{v}}(f)$ have a common factor.
\end{theorem}

\begin{proof}
For any line polynomial $\phi$ in direction $\vec{v}$, and for any polynomial $g$, the $\vec{v}$-fibers of the
product $\phi g$ have a common factor $\phi$.
In other words, if a polynomial $f$ has a line polynomial factor $\phi$ in direction $\vec{v}$ then the polynomials in $\mathcal{F}_{\vec{v}}(f)$ have the normal form of $\phi$ as a common factor.

    For the converse direction, assume that
    the polynomials in $\mathcal{F}_{\vec{v}}(f)$ have a common factor $\phi$ which is thus a line polynomial in direction $\vec{v}$.
    Then there exist vectors $\vec{u}_1, \ldots , \vec{u}_n \in \Z^2$ and polynomials $\phi\psi_1, \ldots , \phi\psi_n \in\mathcal{F}_{\vec{v}}(f)$ such that
    $$
    f = X^{\vec{u}_1} \phi \psi_1 + \ldots + X^{\vec{u}_n} \phi \psi_n.
    $$
    Hence
    $\phi$ is a line polynomial factor of $f$ in direction $\vec{v}$.
\end{proof}
Note that Lemma~\ref{lemma1} actually follows immediately from Theorem~\ref{theorem1}: A vertex instead of an outer edge
in direction $\vec{v}$ or $-\vec{v}$ provides a non-zero monomial $\vec{v}$-fiber, which implies that 
the polynomials in $\mathcal{F}_{\vec{v}}(f)$ have no common factors.

Thus to find out the line polynomial factors of $f$ we first
need to find out the possible directions of the line polynomials, that is, the directions of the (finitely many)
outer edges of $f$, and then we need to check for which of these possible directions $\vec{v}$ the polynomials
in $\mathcal{F}_{\vec{v}}(f)$ have a common factor.
There are clearly algorithms to find the outer edges of a given polynomial
and to determine whether finitely many line polynomials have a common factor.
If such a factor exists then $f$ has a line polynomial factor in this direction by Theorem \ref{theorem1}.
Thus we have proved the following theorem.

\begin{theorem}
    There is an algorithm to find the line polynomial factors of a given (Laurent) polynomial.
\end{theorem}

\section{Perfect coverings}

In this paper a $graph$ is a tuple $G=(V,E)$ where $V$ is the (possibly infinite) 
vertex set of $G$ and $E \subseteq \{ \{ u,v \} \mid u,v \in V, u \neq v \}$ is the edge set of $G$.
Thus the graphs we consider are \emph{simple} and \emph{undirected}. We also assume that all vertices have only finitely many neighbors in the graph.
For a graph $G=(V,E)$ we call any subset $S \subseteq V$ of the vertex set a \emph{code} in $G$.
The distance $d(u,v)$ of two vertices $u,v \in V$ is the length of a shortest path between them.
The \emph{(closed) $r$-neighborhood} of a vertex $u \in V$ is the set $N_r(u) = \{ v \in V \mid d(v,u) \leq r \}$, that is, the ball of radius $r$ centered at $u$.
Let us now give the definition of the family of codes we consider.

\begin{definition} \label{maar1}
    Let $G=(V,E)$ be a graph.
    A code $S \subseteq V$ is an \emph{$(r,b,a)$-covering}
    in $G$ for non-negative integers $b$ and $a$ if the $r$-neighborhood of every vertex in $S$ contains exactly $b$ elements of $S$ and the $r$-neighborhood of every vertex not in $S$ contains exactly $a$ elements of $S$, that is, if for every $u \in V$
    $$
    |N_r(u) \cap S| =
    \begin{cases}
        b \text{ if } u \in S \\
        a \text{ if } u \not \in S
    \end{cases}
    .
    $$
\end{definition}

\noindent
By a \emph{perfect (multiple) covering} we mean any $(r,b,a)$-covering.


\subsection{Infinite grids}

\begin{figure}
	\centering
	\begin{tikzpicture}[scale=0.6]
		\draw[fill=black] (0,0) circle(3pt);
		\draw (-1,0) circle(3pt);
		\draw (0,-1) circle(3pt);
		\draw (1,0) circle(3pt);
		\draw (0,1) circle(3pt);
		
		\draw[densely dotted] (0,1.4) -- (1.4,0) -- (0,-1.4) -- (-1.4,0) -- (0,1.4);
		\node[scale=0.7] at (0,-2) {(a) The square grid};
		
		\draw[fill=black] (5,0) circle(3pt);
		\draw (4,0) circle(3pt);
		\draw (5,-1) circle(3pt);
		\draw (6,0) circle(3pt);
		\draw (5,1) circle(3pt);	
		\draw (6,1) circle(3pt);	
		\draw (6,-1) circle(3pt);
		\draw (4,1) circle(3pt);
		\draw (4,-1) circle(3pt);
		
		\draw[densely dotted] (6.4,1.4) -- (6.4,-1.4) -- (3.6,-1.4) -- (3.6,1.4) -- (6.4,1.4);
		\node[scale=0.7] at (5,-2) {(b) The king grid};
		
		\draw[fill=black] (10,0) circle(3pt);
		\draw (9,0) circle(3pt);
		\draw (10,-1) circle(3pt);
		\draw (11,0) circle(3pt);
		\draw (10,1) circle(3pt);	
		\draw (11,1) circle(3pt);	
		\draw (9,-1) circle(3pt);
		
		\draw[densely dotted] (11.4,1.4) -- (11.4,0) -- (10,-1.4) -- (8.6,-1.4) -- (8.6,0) -- (10,1.4) -- (11.4,1.4);
		\node[scale=0.7] at (10,-2) {(c) The triangular grid};
	\end{tikzpicture}
	\caption{The $1$-neighborhoods of the black vertex in (a) the square grid, (b) the king grid, and (c) the triangular grid.}
    \label{Neighborhoods}
\end{figure}
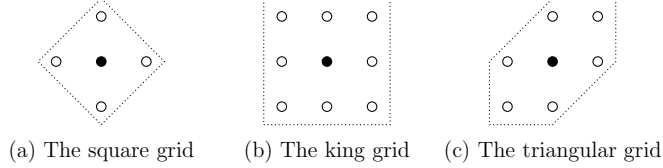

\noindent
An \emph{infinite grid} is a translation invariant 
graph with the vertex set $\Z^2$.
In other words, in infinite grids we have $N_r(\vec{u}) = \vec{u} + N_r(\vec{0})$ for all $\vec{u} \in \Z^2$.
The \emph{square grid} is the graph
$
(\Z^2, E_{\mathcal{S}})$ with
$E_{\mathcal{S}} = \{ \{ \vec{u} , \vec{v} \} \mid \vec{u} - \vec{v} \in \{  (\pm 1,0), (0,\pm 1) \} \}$, the \emph{king grid} is the graph
$
(\Z^2, E_{\mathcal{K}})$ with
$E_{\mathcal{K}} = \{ \{ \vec{u}, \vec{v} \} \mid \vec{u} - \vec{v} \in \{ (\pm 1,0),(0,\pm 1),(\pm 1,\pm 1) \}  \}$ and
the \emph{triangular grid} is the graph $
(\Z^2, E_{\mathcal{T}})$ with
$
E_{\mathcal{T}} = \{ \{ \vec{u}, \vec{v} \} \mid \vec{u} - \vec{v} \in \{ (\pm 1,0),(0,\pm 1),(1,1),(-1,-1) \}  \}.
$
See Figure \ref{Neighborhoods} for the 1-neighborhoods of a vertex in these graphs.
A code $S \subseteq \Z^2$ is periodic if $S = S + \vec{t}$ for some non-zero $\vec{t} \in \Z^2$.
It is two-periodic if $S = S + \vec{t}_1$ and $S = S + \vec{t}_2$ where $\vec{t}_1$ and $\vec{t}_2$ are linearly independent.
The following result is by Axenovich.

\begin{theorem}[\cite{Axenovich}] \label{squaregrid1}
    If $b-a \neq 1$ then any $(1,b,a)$-covering in the square grid is two-periodic.
\end{theorem}

\noindent
A code $S \subseteq \Z^2$ in any infinite grid can be presented as a configuration $c \in \{0,1\}^{\Z^2}$ which is defined such that $c_{\vec{u}} = 1$ if $\vec{u} \in S$ and $c_{\vec{u}} = 0$ if $\vec{u} \not \in S$.
The positioning of the codewords in the $r$-neighborhood of any vertex $\vec{u} \in \Z^2$
is then presented as a finite pattern $c|_{\vec{u} + N_r(\vec{0})}$.

\begin{definition} \label{maar2}
    A configuration $c \in \{ 0,1 \}^{\Z^2}$ is a \emph{$(D,b,a)$-covering}
    for a finite shape $D \subseteq \Z^2$ (the neighborhood) and non-negative integers $b$ and $a$ (the covering constants) if for all $\vec{u} \in \Z^2$ the pattern $c|_{\vec{u} + D}$ contains exactly $b$ symbols 1 if $c_{\vec{u}} = 1$ and exactly $a$ symbols 1 if $c_{\vec{u}} = 0$.
\end{definition}

\noindent
We call also any $(D,b,a)$-covering perfect and hence a perfect covering is either a code in a graph or a two-dimensional binary configuration. 

Definitions \ref{maar1} and \ref{maar2} are consistent in infinite grids:
a code $S$ in an infinite grid $G$
is an $(r,b,a)$-covering if and only if the configuration $c \in \{ 0,1 \}^{\Z^2}$ presenting $S$ is a $(D,b,a)$-covering where $D$ is the $r$-neighborhood of $\vec{0}$ in $G$.
For a set $D \subseteq \Z^2$ we define its \emph{characteristic polynomial} to be $f_D(X) = \sum _{\vec{u} \in D} X^{-\vec{u}}$.
Let us denote by $\mathbbm{1}(X)$
the constant power series $\sum_{\vec{u} \in \Z^2} X^{\vec{u}}$.
If $c$ is a $(D,b,a)$-covering then from the definition we get
that
$
f_D(X)c(X) = (b-a)c(X) + a \mathbbm{1}(X)
$
which is equivalent
to
$\left (f_D(X) - (b-a) \right )c(X) = a \mathbbm{1}(X)$.
Thus if $c$ is a $(D,b,a)$-covering then $f_D(X) - (b-a) \in \Per(c)$.
Using our formulation we get a simple proof for Theorem \ref{squaregrid1}:

\setcounter{reformulation}{7}
\begin{reformulation}
    Let $D$ be the 1-neighborhood of $\vec{0}$ in the square grid and assume that $b-a \neq 1$.
    Then every $(D,b,a)$-covering is two-periodic.
\end{reformulation}

\begin{proof}
	Let $c$ be an arbitrary $(D,b,a)$-covering.
    We show that $g = f_D - (b-a) = x^{-1} + y^{-1} + 1 - (b-a) + x + y  \in \Per(c)$ has no line polynomial factors. Then $c$ is two-periodic by Theorem \ref{theorem on line polynomial factors}.
    The outer edges of $g$ are in directions $(1,1),(-1,-1),(1,-1)$ and $(-1,1)$ and hence
    by Lemma~\ref{lemma1}
    any line polynomial factor of $g$ is either in direction $(1,1)$ or $(1,-1)$.
    For $\vec{v} \in \{(1,1),(1,-1)\}$
    we have $\mathcal{F}_{\vec{v}}(g) = \{ 1+t,1 -(b-a) \}$.
    See Figure \ref{Illustrations} for an illustration. Since $1 -(b-a)$ is a non-trivial monomial,
    by Theorem \ref{theorem1} the periodizer $g \in \Per(c)$ has no line polynomial factors.
\end{proof}

\noindent
The following result was already proved in a more general form in \cite{puzynina2}.
We give a short proof using our algebraic approach.

\begin{theorem}[\cite{puzynina2}] \label{squaregrid2}
    Let $r \geq 2$ and let $D$ be the $r$-neighborhood of $\vec{0}$ in the square grid.
    Then every $(D,b,a)$-covering is two-periodic.
    In other words, all $(r,b,a)$-coverings in the square grid are two-periodic for all $r \geq 2$.
\end{theorem}

\begin{proof}
    Let $c$ be an arbitrary $(D,b,a)$-covering.
    Again, by Theorem \ref{theorem on line polynomial factors}, it is enough to
    show that $g = f_D -(b-a) \in \Per(c)$ has no line polynomial factors.
    By Lemma~\ref{lemma1} any line polynomial factor of $g$
    has direction
    $(1,1)$ or $(1,-1)$.
    So assume that $\vec{v} \in \{ (1,1), (1,-1) \}$.
    We have $\phi_1=1+t+\ldots+t^r \in \mathcal{F}_{\vec{v}}(g)$ and
    $\phi_2=1+t+\ldots+t^{r-1} \in \mathcal{F}_{\vec{v}}(g)$.
    See Figure \ref{Illustrations} for an illustration in the case $r=2$. Since $\phi_1-\phi_2=t^r$, the polynomials
    $\phi_1$ and $\phi_2$ have no common factors, and hence by Theorem \ref{theorem1} the periodizer
    $g$ has no line polynomial factors.
\end{proof}

\begin{figure}
	\centering
	\begin{tikzpicture}[scale=0.5]
		\draw[fill=red] (0,0) circle(3pt);
		\draw[fill=blue] (-1,0) circle(3pt);
		\draw[] (0,-1) circle(3pt);
		\draw[] (1,0) circle(3pt);
		\draw[fill=blue] (0,1) circle(3pt);
		
		\draw (-1.5,-0.5) -- (0.5,1.5);
		\draw (-0.5,-0.5) -- (0.5,0.5);
		\node[rotate=45] at (0.4,1.7) {\tiny $1+t$};	
		\node[rotate=45] at (0.4,0.8) {\tiny $1-(b-a)$};	
		
		

	
		\draw[] (6,0) circle(3pt);
		\draw[fill=red] (5,0) circle(3pt);
		\draw[] (6,-1) circle(3pt);
		\draw[] (7,0) circle(3pt);
		\draw[fill=red] (6,1) circle(3pt);
		\draw[fill=blue] (6,2) circle(3pt);
		\draw[] (6,-2) circle(3pt);
		\draw[] (7,1) circle(3pt);
		\draw[] (7,-1) circle(3pt);
		\draw[fill=blue] (5,1) circle(3pt);
		\draw[] (5,-1) circle(3pt);
		\draw[fill=blue] (4,0) circle(3pt);
		\draw[] (8,0) circle(3pt);
		
		\draw (3.5,-0.5) -- (6.5,2.5);
		\draw (4.5,-0.5) -- (6.5,1.5);
		
		\node[rotate=45] at (5,1.4) {\tiny $1+t+t^2$};
		\node[rotate=45] at (6.5,1.8) {\tiny $1+t$};

		
		
		\draw[fill=red] (13,0) circle(3pt);
		\draw[fill=red] (12,0) circle(3pt);
		\draw[] (13,-1) circle(3pt);
		\draw[fill=red] (14,0) circle(3pt);
		\draw[] (13,1) circle(3pt);
		\draw[fill=blue] (13,2) circle(3pt);
		\draw[] (13,-2) circle(3pt);
		\draw[] (14,1) circle(3pt);
		\draw[] (14,-1) circle(3pt);
		\draw[] (12,1) circle(3pt);
		\draw[] (12,-1) circle(3pt);
		\draw[fill=red] (11,0) circle(3pt);
		\draw[fill=red] (15,0) circle(3pt);
		\draw[] (15,1) circle(3pt);
		\draw[fill=blue] (15,2) circle(3pt);
		\draw[] (15,-1) circle(3pt);
		\draw[] (15,-2) circle(3pt);
		\draw[] (11,1) circle(3pt);
		\draw[fill=blue] (11,2) circle(3pt);
		\draw[] (11,-1) circle(3pt);
		\draw[] (11,-2) circle(3pt);
		\draw[fill=blue] (12,2) circle(3pt);
		\draw[] (12,-2) circle(3pt);
		\draw[fill=blue] (14,2) circle(3pt);
		\draw[] (14,-2) circle(3pt);
		
		\draw (10.5,2) -- (15.5,2);
		\draw (10.5,0) -- (15.5,0);
		\node at (14,2.5) {\tiny $1+t+t^2+t^3+t^4$};
		\node at (14,0.5) {\tiny $1+t+(1-(b-a))t^2+t^3+t^4$};

	\end{tikzpicture}
	\caption{The constellation on the left illustrates the proof of Theorem \ref{squaregrid1}, the constellation on the center illustrates the proof of Theorem \ref{squaregrid2} with $r=2$ and the constellation on the right illustrates the proof of Theorem \ref{kinggrid} with $r=2$.}
    \label{Illustrations}
\end{figure}
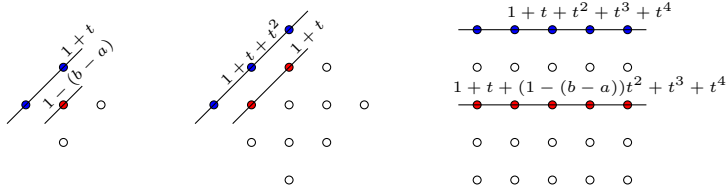

\noindent
If $a \neq b$ then for all $r \geq 1$ any $(r,b,a)$-covering in the king grid is two-periodic:

\begin{theorem} \label{kinggrid}
    Let $r \geq 1$ be arbitrary and let $D$ be the $r$-neighborhood of $\vec{0}$ in the king grid and assume that $a \neq b$.
    Then any $(D,b,a)$-covering is two-periodic.
    In other words, all $(r,b,a)$-coverings in the king grid are two-periodic whenever $a \neq b$.
\end{theorem}

\begin{proof}
    Let $c$ be an arbitrary $(D,b,a)$-covering.
    By Theorem \ref{theorem on line polynomial factors} it is sufficient to show that $g = f_D -(b-a)$ has no line polynomial factors.
    The outer edges of $g$
    are in directions $(1,0),(-1,0),(0,1)$ and $(0,-1)$.
    Hence
    by Lemma~\ref{lemma1} any line polynomial factor of $g$
    has direction
    $(1,0)$
    or
    $(0,1)$.
    Let $\vec{v} \in \{ (1,0),(0,1)\}$.
    We have
    $\phi_1 = 1+t+\ldots + t^{r-1} + (1-(b-a))t^r +t^{r+1} + \ldots + t^{2r} \in \mathcal{F}_{\vec{v}}(g)$
    and
    $\phi_2 = 1+t+ \ldots + t^{2r} \in \mathcal{F}_{\vec{v}}(g)$.
    See Figure \ref{Illustrations} for an illustration with $r=2$.
    Since $\phi_2-\phi_1=(b-a)t^r$ is a non-trivial monomial, $\phi_1$ and $\phi_2$ have no common factors.
    Thus $g$ has no line polynomial factors by Theorem \ref{theorem1}.
\end{proof}

\noindent
Similarly as in the square grid we can give simple proofs for known results from \cite{puzynina2} concerning forced periodicity in the triangular grid:

\begin{theorem}[\cite{puzynina2}] \label{triangulargrid1}
    Let $D$ be the 1-neighborhood of $\vec{0}$ in the triangular grid and assume that
    $b-a \neq -1$.
    Then every $(D,b,a)$-covering in the triangular grid is two-periodic.
    In other words, all $(1,b,a)$-coverings in the triangular grid are two-periodic whenever $b-a \neq -1$.
\end{theorem}

\begin{theorem}[\cite{puzynina2}] \label{triangulargrid2}
    Let $r \geq 2$ and let $D$ be the $r$-neighborhood of $\vec{0}$ in the triangular grid.
    Then every $(D,b,a)$-covering is two-periodic.
    In other words, all $(r,b,a)$-coverings in the triangular grid are two-periodic for $r \geq 2$.
\end{theorem}

\subsection{General convex neighborhoods}

A shape $D \subseteq \Z^2$ is \emph{convex} if it is the intersection $D = \conv(D) \cap \Z^2$ where $\conv(D) \subseteq \R^2$ is the real convex hull of $D$.

Let $D \subseteq \Z^2$ be a finite convex shape.
Any $(D,b,a)$-covering has a periodizer $g = f_D - (b-a)$.
As earlier, we study whether $g$ has any line polynomial factors.
For any $\vec{v} \neq 0$ the set $\mathcal{F}_{\vec{v}}(f_D)$ contains only polynomials $\phi_n = 1 + \ldots + t^{n-1}$ for different $n \geq 1$ since $D$ is convex: if $D$ contains two points then $D$ contains every point between them.
Thus $\mathcal{F}_{\vec{v}}(g)$ contains only polynomials $\phi_n$ for different $n \geq 1$ and, if $b-a \neq 0$, also a polynomial $\phi_{n_0} - (b-a)t^{m_0}$ for some $n_0 \geq 1$ such that $\phi_{n_0} \in \mathcal{F}_{\vec{v}}(f_D)$ and for some $m_0 \geq 0$.
If $b-a=0$ then $g=f_D$ and thus $\mathcal{F}_{\vec{v}}(g) = \mathcal{F}_{\vec{v}}(f_D)$.

Two polynomials $\phi_m$ and $\phi_n$ have a common factor if and only if $\gcd(m,n) > 1$.
More generally, the polynomials $\phi_{n_1}, \ldots, \phi_{n_r}$ have a common factor if and only if $d = \gcd(n_1,\ldots,n_r) > 1$ and, in fact, their greatest common factor is the $d$th \emph{cyclotomic polynomial}
$$
\prod_{\substack{1 \leq k \leq d \\ \gcd(k,d) = 1}}(t - e^{i \cdot \frac{2 \pi k}{d}}).
$$

Let us introduce the following notation. For any polynomial $f$, we
denote by $\mathcal{F}'_{\vec{v}}(f)$ the set of normal forms of the
non-zero fibers $\sum_{k \in \Z} f_{\vec{u} + k \vec{v}} X^{\vec{u} + k \vec{v}}$ for all $\vec{u}\not\in\Z\vec{v}$.
In other words, we exclude the fiber through the origin. Let us also denote $\fib{\vec{v}}{f}$ for the
normal form of the fiber $\sum_{k \in \Z} f_{k \vec{v}} X^{k \vec{v}}$ through the origin. 
We have
$\mathcal{F}_{\vec{v}}(f)=\mathcal{F}'_{\vec{v}}(f)\cup\{\fib{\vec{v}}{f}\}$ if $\fib{\vec{v}}{f} \neq 0$ and $\mathcal{F}_{\vec{v}}(f)=\mathcal{F}'_{\vec{v}}(f)$ if $\fib{\vec{v}}{f} = 0$.

Applying Theorems \ref{theorem on line polynomial factors} and \ref{theorem1} we have the following theorem that gives sufficient conditions for every $(D,b,a)$-covering to be periodic for a finite and convex $D$.
The first part of the theorem was also mentioned in \cite{geravker-puzynina} in a more general form.

\begin{theorem}
	Let $D$ be a finite convex shape, $g = f_D - (b-a)$
    and let $E$ be the set of the outer edge directions of $g$.
    \begin{itemize}
    		\item Assume that $b-a = 0$. For any $\vec{v} \in E$ denote $d_\vec{v}=\gcd(n_1,\ldots,n_r)$ where
    $\mathcal{F}_{\vec{v}}(g) = \{ \phi_{n_1},\ldots,\phi_{n_r}\}$. If $d_\vec{v} = 1$ holds for all
    $\vec{v} \in E$ then every $(D,b,a)$-covering is two-periodic. If $d_\vec{v} = 1$ holds for all but some
    parallel $\vec{v} \in E$ then every $(D,b,a)$-covering is periodic.
    		\item Assume that $b-a \neq 0$. For any $\vec{v} \in E$ denote $d_\vec{v}=\gcd(n_1,\ldots,n_r)$ where
    $\mathcal{F}'_{\vec{v}}(g) = \{ \phi_{n_1},\ldots,\phi_{n_r}\}$.
If the $d_\vec{v}$'th cyclotomic polynomial and $\fib{\vec{v}}{g}$ have no common factors for
any $\vec{v} \in E$ then every $(D,b,a)$-covering is two-periodic. If the condition holds for all but some
    parallel $\vec{v} \in E$ then every $(D,b,a)$-covering is periodic. (Note that the condition is satisfied, in particular, if $d_\vec{v} = 1$.)
    \end{itemize}
\end{theorem}

\section{Algorithmic aspects}


All coverings are periodic, in particular, if there are no coverings at all!
It is useful to be able to detect such trivial cases.

The set
$$
\mathcal{S}(D,b,a) =
\{ c \in \{0,1\}^{\Z^2} \mid (f_D - (b-a)) c = a \mathbbm{1}(X) \}
$$
of all $(D,b,a)$-coverings is an SFT for any given finite shape $D$ and non-negative integers $b$ and $a$.
Hence the question whether there exist any $(D,b,a)$-coverings for given neighborhood
$D$ and covering constants $b$ and $a$ is equivalent to the question whether
the SFT $\mathcal{S} = \mathcal{S}(D,b,a)$ is non-empty. The question of emptiness
of a given SFT is in general undecidable, but if the SFT is known to be not aperiodic then
the problem becomes decidable.
In particular, if $g = f_D - (b-a)$ has line polynomial factors in at most one direction then
this question is decidable:



\begin{theorem}
\label{thm:effective}
Let finite $D \subseteq \Z^2$ and non-negative integers $b$ and $a$ be given such that
the polynomial $g=f_D - (b-a)$ has line polynomial factors in at most one parallel direction.
Then there exists an algorithm to determine whether there exist any $(D,b,a)$-coverings.
\end{theorem}

\begin{proof}	
Let $\mathcal{S}=\mathcal{S}(D,b,a)$ be the SFT of all $(D,b,a)$-coverings.
Since $g$ has line polynomial factors in at most one direction, by Theorem \ref{theorem on line polynomial factors}
every element of $\mathcal{S}$ is periodic.
Any two-dimensional SFT that contains periodic configurations contains also two-periodic configurations, so
$\mathcal{S}$ is either empty or contains a two-periodic configuration. By a standard argumentation by H. Wang \cite{wang}
there exist semi-algorithms to determine whether a given SFT is empty and whether a given SFT contains a two-periodic configuration.	Running these two semi-algorithms in parallel gives us an algorithm to test whether $\mathcal{S} \neq \emptyset$.
\end{proof}

\noindent
One may also want to design a perfect $(D,b,a)$-covering for given $D$, $b$ and $a$.
This can be effectively done under the assumptions of Theorem~\ref{thm:effective}:
As we have seen, if
$\mathcal{S}=\mathcal{S}(D,b,a)$ is non-empty it contains a two-periodic configuration. For any two-periodic
configuration $c$ it is easy to check if $c$ contains a forbidden pattern.
By enumerating two-periodic configurations one-by-one one is guaranteed to find eventually one that is in $\mathcal{S}$.

If the polynomial $g$ has no line polynomial factors then the following stronger result holds:

\begin{theorem} \label{effective}
	If the polynomial $g=f_D - (b-a)$ has no line polynomial factors
	for given finite shape $D \subseteq \Z^2$ and non-negative integers $b$ and $a$	
	 then the SFT $\mathcal{S} = \mathcal{S}(D,b,a)$ is finite. One can then effectively construct
all the finitely many elements of $\mathcal{S}$.
\end{theorem}

\noindent
The proof of the first part of above theorem relies on the fact that
a two-dimensional subshift is finite if and only if it contains only two-periodic cofigurations~\cite{ballier}.
If $g$ has no line polynomial factors then every configuration it periodizes (including every configuration in $\mathcal{S}$) is two-periodic by Theorem \ref{theorem on line polynomial factors}, and hence $\mathcal{S}$ is finite.
The ``moreover'' part of the theorem, {\it i.e.},  the fact that one can
effectively produce all the finitely many elements of $\mathcal{S}$
holds generally for finite  SFTs. 
(The proof is provided in the Appendix for the sake of completeness.)





\section*{References}

\bibliographystyle{plain}
\bibliography{Biblio}

\newpage
\section*{Appendix}

\subsection*{Proofs of Theorems \ref{triangulargrid1} and \ref{triangulargrid2}}

\begin{manualtheorem}{\ref{triangulargrid1}.}
    Let $D$ be the 1-neighborhood of $\vec{0}$ in the triangular grid and assume that
    $b-a \neq -1$.
    Then every $(D,b,a)$-covering in the triangular grid is two-periodic.
    In other words, all $(1,b,a)$-coverings in the triangular grid are two-periodic whenever $b-a \neq -1$.
\end{manualtheorem}

\begin{proof}
Let $c$ be an arbitrary $(D,b,a)$-covering.
    Once again, we show that $g = f_D -(b-a) = x^{-1}y^{-1} + x^{-1} + y^{-1} + 1-(b-a) + x + y + xy$ has no line polynomial factors, so that by Theorem \ref{theorem on line polynomial factors} the configuration $c$ is two-periodic.
    The outer edges of $g$ have directions
    $(1,1),(-1,-1),(1,0),(-1,0),(0,1)$ and $(0,-1)$
    and hence by Lemma~\ref{lemma1} any line polynomial factor of $g$ has direction $(1,1)$, $(1,0)$ or $(0,1)$.
    So, let $\vec{v} \in \{ (1,1), (1,0),(0,1) \}$.
    We have
    $\mathcal{F}_{\vec{v}}(g) = \{ 1+t, 1+ (1 - (b-a))t + t^2 \}$.
    See Figure \ref{Illustrations2} for an illustration.
    Polynomials $\phi_1=1+t$ and $\phi_2=1+ (1 - (b-a))t + t^2$ satisfy $\phi_1^2-\phi_2= (1+b-a)t$ so that
    they do not have any common factors if $b-a \neq -1$.
    Thus
    $g$ has no line polynomial factors by Theorem \ref{theorem1}.
\end{proof}

\bigskip

\begin{manualtheorem}{ \ref{triangulargrid2}}
    Let $r \geq 2$ and let $D$ be the $r$-neighborhood of $\vec{0}$ in the triangular grid.
    Then every $(D,b,a)$-covering is two-periodic.
    In other words, all $(r,b,a)$-coverings in the triangular grid are two-periodic for $r \geq 2$.
\end{manualtheorem}

\begin{proof}
Let $c$ be an arbitrary $(D,b,a)$-covering.
    We show that $g = f_D -(b-a) \in \Per(c)$ has no line polynomial factors, which by Theorem \ref{theorem on line polynomial factors} implies that the configuration $c$ is two-periodic.
    The outer edges
    of $g$ have directions
    $(1,1)$, $(-1,-1)$, $(1,0)$, $(-1,0)$, $(0,1)$ and $(0,-1)$,
    and hence by Lemma~\ref{lemma1} any line polynomial factor of $g$
    has direction $(1,1)$, $(1,0)$ or $(0,1)$.
    So, let $\vec{v} \in \{ (1,1), (1,0),(0,1) \}$.
    There exists $n \geq 1$ such that $1+t+ \ldots + t^n \in \mathcal{F}_{\vec{v}}(g)$ and $1+t+ \ldots + t^{n+1} \in \mathcal{F}_{\vec{v}}(g)$.
    See Figure \ref{Illustrations2} for an illustration with $r=2$.
    Since these two polynomials have no common factors
    $g$ has no line polynomial factors by Theorem \ref{theorem1}. 
\end{proof}

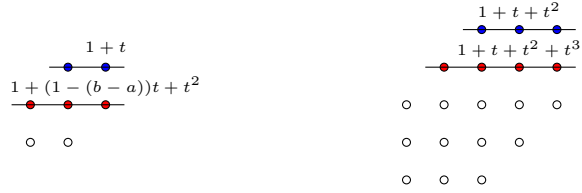
\begin{figure}
	\centering
	\begin{tikzpicture}[scale=0.5]
		\draw[fill=red] (0,-6) circle(3pt);
		\draw[fill=red] (-1,-6) circle(3pt);
		\draw[] (0,-7) circle(3pt);
		\draw[fill=red] (1,-6) circle(3pt);
		\draw[fill=blue] (0,-5) circle(3pt);
		\draw[fill=blue] (1,-5) circle(3pt);
		\draw[] (-1,-7) circle(3pt);
	
		\draw (-0.5,-5) -- (1.5,-5);
		\draw (-1.5,-6) -- (1.5,-6);
		\node at (1,-4.5) {\tiny $1+t$};
		\node at (1,-5.5) {\tiny $1+(1-(b-a))t+t^2$};
	
	
		\draw[] (11,-6) circle(3pt);
		\draw[] (10,-6) circle(3pt);
		\draw[] (11,-7) circle(3pt);
		\draw[] (12,-6) circle(3pt);
		\draw[fill=red] (11,-5) circle(3pt);
		\draw[fill=blue] (11,-4) circle(3pt);
		\draw[] (11,-8) circle(3pt);
		\draw[fill=red] (12,-5) circle(3pt);
		\draw[] (12,-7) circle(3pt);
		\draw[fill=red] (10,-5) circle(3pt);
		\draw[] (10,-7) circle(3pt);
		\draw[] (9,-6) circle(3pt);
		\draw[] (13,-6) circle(3pt);
		\draw[] (9,-7) circle(3pt);
		\draw[] (9,-8) circle(3pt);
		\draw[] (10,-8) circle(3pt);
		\draw[fill=blue] (12,-4) circle(3pt);
		\draw[fill=blue] (13,-4) circle(3pt);
		\draw[fill=red] (13,-5) circle(3pt);
		
		\draw (10.5,-4) -- (13.5,-4);
		\draw (9.5,-5) -- (13.5,-5);
		
		\node at (12,-3.5) {\tiny $1+t+t^2$};
		\node at (12,-4.5) {\tiny $1+t+t^2+t^3$};

	\end{tikzpicture}
	\caption{The constellation on the left illustrates the proof of Theorem \ref{triangulargrid1} and the constellation on the right illustrates the proof of Theorem \ref{triangulargrid2} with $r=2$.}
    \label{Illustrations2}
\end{figure}

\subsection*{An alternative proof of Theorem~\ref{theorem on line polynomial factors}}

\begin{manualtheorem}{\ref{theorem on line polynomial factors}.}
    Let $c$ be a two-dimensional configuration and $f \in \Per(c)$. Then the following conditions hold.
    \begin{itemize}
        \item If $f$ does not have any line polynomial factors then $c$ is two-periodic.
        \item If all line polynomial factors of $f$ are in the same direction then $c$ is periodic in this direction.
    \end{itemize}
\end{manualtheorem}

\noindent
\emph{Second proof sketch.}
The existence of a non-trivial periodizer $f$ implies
by Theorem~\ref{special annihilator} that $c$ has a special annihilator
$g=\phi_1 \cdots \phi_m$ that is a product of (difference) line polynomials $\phi_1, \ldots ,\phi_m$ in pairwise different
directions.
All irreducible factors of $g$ are line polynomials. If $f$ does not have any line polynomial factors then
the periodizers $f$ and $g$ do not have common factors.
We can assume that both are proper polynomials as
they can be multiplied by a monomial if needed. The \emph{$x$-resultant} of $f,g\in \C[x,y]$
is a polynomial $\Res_x(f,g)=\alpha f+\beta g$ for some $\alpha,\beta\in \C[x,y]$ such that the variable $x$ is
eliminated, {\it i.e.},
$\Res_x(f,g)$ is a polynomial in variable $y$ only. Moreover, since $f$ and $g$ do not have common factors,
$\Res_x(f,g)$ is not identically zero. Because $f,g\in\Per(c)$ also $\Res_x(f,g)\in \Per(c)$, implying that
$c$ has a non-trivial annihilator containing only variable $y$. This means that $c$ is periodic in the vertical direction.
Analogously, the \emph{$y$-resultant} $\Res_y(f,g)$ shows that $c$ is horizontally periodic, and hence two-periodic.

The proof for the case that $f$ has line polynomial factors only in one direction $\vec{v}$
goes analogously by considering $\phi c$
instead of $c$, where $\phi$ is the greatest common line polynomial factor of $f$ and $g$ in the direction $\vec{v}$.
We get that $\phi c$ is two-periodic, implying that $c$ is periodic in the direction $\vec{v}$.
\qed

\subsection*{An algorithm to find all elements of a given finite SFT}

\begin{theorem}
Given a finite $F \subseteq \A^*$ such that $X_F$ is finite, one can effectively construct the elements of $X_F$.
\end{theorem}

\begin{proof}	
Given a finite $F \subseteq \A^*$ and a pattern $p\in \A^D$, assuming that strongly periodic configurations are dense in $X_F$,
one can effectively check whether $p\in \Lang{X_F}$. Indeed, we have a semi-algorithm for the positive instances that
guesses a strongly periodic configuration $c$ and verifies that $c\in X_F$ and $p\in\Lang{c}$. A semi-algorithm for the
negative instances exists for any SFT $X_F$ and is a standard compactness argument: guess a finite $E\subseteq \Z^d$ such that
$D\subseteq E$ and verify that every $q\in \A^E$ such that $q|_D=p$ contains a forbidden subpattern.

Consequently, given finite $F,G \subseteq \A^*$,  assuming that strongly periodic configurations are dense in $X_F$ and $X_G$,
one can effectively determine whether $X_F=X_G$. Indeed, $X_F\subseteq X_G$ if and only if no $p\in G$ is in $\Lang{X_F}$, a
condition that we have shown above to be decidable. Analogously we can test $X_G\subseteq X_F$.

Finally, let a finite $F \subseteq \A^*$ be given such that $X_F$ is known to be finite. All elements of $X_F$ are strongly periodic so that strongly periodic configurations are certainly dense in $X_F$. One can effectively enumerate
all finite sets $P$ of strongly periodic configurations. For each $P$ that is translation invariant
(and hence a finite SFT) one can construct a finite set $G\subseteq \A^*$ of forbidden patterns such that $X_G=P$.
As shown above, there is an algorithm to test
whether $X_F=X_G=P$. Since $X_F$ is finite, a set $P$ is eventually found such that $X_F=P$.
\end{proof}

\end{document}